\documentclass[a4paper]{amsart}
\usepackage{amsmath,amssymb,times, amscd, graphicx, xspace,mathrsfs,a4wide,hyperref,url}
\hypersetup{
  colorlinks   = true, %Colours links instead of ugly boxes
  urlcolor     = blue, %Colour for external hyperlinks
  linkcolor    = blue, %Colour of internal links
  citecolor   = red %Colour of citations
}

\newcommand{\Z}{\ensuremath{\mathbb{Z}}\xspace}
\newcommand{\Q}{\ensuremath{\mathbb{Q}}\xspace}
\newcommand{\R}{\ensuremath{\mathbb{R}}\xspace}
\newcommand{\C}{\ensuremath{\mathbb{C}}\xspace}
\newcommand{\A}{\ensuremath{\mathbb{A}}\xspace}
\newcommand{\F}{\ensuremath{\mathbb{F}}\xspace}
\newcommand{\Sp}{\mathrm{Sp}\xspace}

\newcommand{\m}{\ensuremath{\mathfrak{m}}\xspace}

\newcommand{\n}{\ensuremath{\mathfrak{n}}\xspace}

\newcommand{\q}{\ensuremath{\mathfrak{q}}\xspace}
\newcommand{\lf}{\ensuremath{\mathfrak{l}}\xspace}

\newcommand{\OO}{\ensuremath{\mathcal{O}}\xspace}

\newcommand{\Nm}{\mathrm{Nm}\xspace}

\newcommand{\comment}[1]{}

\DeclareMathOperator{\Gal}{Gal}

\DeclareMathOperator{\Hom}{Hom}

\newcommand{\rhobar}{\ensuremath{\overline{\rho}}\xspace}
\newcommand{\T}{\ensuremath{\mathbb{T}}\xspace}

\newcommand{\GL}{\ensuremath{\mathrm{GL}}\xspace}
\newcommand{\FF}{\ensuremath{\mathscr{F}}\xspace}

\newcommand{\G}{\ensuremath{\mathbb{G}}\xspace}

\newcommand{\af}{\mathfrak{a}\xspace}

\newcommand{\pf}{\ensuremath{\mathfrak{p}}\xspace}

\newcommand{\cH}{\mathscr{H}\xspace}

\newtheorem{theorem}[subsection]{Theorem}
\newtheorem{proposition}[subsection]{Proposition}
\newtheorem{corollary}[subsection]{Corollary}
\newtheorem{lemma}[subsection]{Lemma}

\theoremstyle{definition}
\newtheorem{definition}[subsection]{Definition}

\newtheorem{remark}[subsection]{Remark}

\mathchardef\mhyphen="2D
\title{Towards local--global compatibility for Hilbert modular forms of low weight}
\author{James Newton}
\input xy
\xyoption{all}
\begin{document}
\bibliographystyle{abbrv}
\begin{abstract}
We prove some new cases of local--global compatibility for the Galois representations associated to Hilbert modular forms of low weight (that is, partial weight one).
\end{abstract}
\maketitle
\section{Introduction}
In this paper, we study the problem of local--global compatibility for the Galois representations attached to Hilbert modular forms of low weight by Jarvis \cite{JarGalRep}. We begin by recalling the main theorem of that paper. Let $F$ be a totally real finite extension of $\Q$. For a prime $p$, let $\iota$ be an isomorphism from a fixed algebraic closure $\overline{\Q}_p$ of $\Q_p$ to $\C$. If $v$ is a finite place of $F$ and $\pi_v$ is an admissible smooth irreducible representation of $\GL_2(F_v)$ over $\C$ we denote the Frobenius-semisimple Weil--Deligne representation associated to $\pi_v$ under the (Hecke normalised) local Langlands correspondence by $\sigma(\pi_v)$. We denote by $\sigma^{\iota}(\pi_v)$ the Weil--Deligne representation with coefficients in $\overline{\Q}_p$ obtained by composition with $\iota^{-1}$. For a continuous $\overline{\Q}_p$-representation $\rho$ of $\Gal(\overline{F}_v/F_v)$, with $v\nmid p$ a finite place of $F$, we denote by $WD(\rho)$ the Weil--Deligne representation associated to $\rho$ by Grothendieck's $p$-adic monodromy theorem, with Frobenius semisimplification $WD(\rho)^{F\mhyphen ss}$ and semisimplification $WD(\rho)^{ss}$ (i.e. we set the monodromy operator to zero to obtain $WD(\rho)^{ss}$).

Results of Carayol \cite{Car}, Blasius and Rogawski \cite{MR1235020}, Rogawski and Tunnell \cite{RogTun},  Taylor \cite{Ta} and Jarvis \cite{JarGalRep} have shown the following:
\begin{theorem}\label{Jarlg}
Let $\pi$ be a cuspidal algebraic\footnote{By algebraic, we just mean that the usual parity condition on the weights holds.} automorphic representation of $\GL_2(\A_F)$, such that for infinite place $\tau$ of $F$ the local factor $\pi_\tau$ is either discrete series or holomorphic limit of discrete series. Then there exists an irreducible representation
\[r_{p,\iota}(\pi): \Gal(\overline{F}/F)\rightarrow \GL_2(\overline{\Q}_p)\] such that if $v$ is a finite place of $F$, with $v\nmid p$, and one of the following holds:
\begin{itemize}
\item $\pi_\tau$ is discrete series for all infinite places $\tau$
\item $\pi_v$ is not special (i.e. $\pi_v$ is not a twist of the Steinberg representation)
\end{itemize}then 
\[WD(r_{p,\iota}(\pi)|_{\Gal(\overline{F}_v/F_v)})^{F\mhyphen ss}\cong\sigma^{\iota}(\pi_v).\]
\end{theorem}
\begin{remark}
The irreducibility of the Galois representation appearing in this theorem is proved using an argument of Ribet \cite{RibetIrred} (see \cite[Proposition 3.1]{TaylorII}).
\end{remark}
\begin{remark}\label{weaklgc}
In the excluded case, where $\pi_\tau$ is holomorphic limit of discrete series for some $\tau$ and $\pi_v$ is special, then Jarvis \cite{JarGalRep} proves that \[WD(r_{p,\iota}(\pi)|_{\Gal(\overline{F}_v/F_v)})^{ss}\cong\sigma^{\iota}(\pi_v)^{ss}.\]
\end{remark}

The main result of this paper is to address some of the excluded cases in Theorem \ref{Jarlg}. We prove:

\begin{theorem}\label{mylgc}
Let $\pi$ be as in Theorem \ref{Jarlg}, suppose that $\pi_\tau$ is holomorphic limit of discrete series for some infinite place $\tau$ and let $v\nmid p$ be a finite place of $F$ such that $\pi_v$ is special.

Suppose the following technical hypotheses hold:
\begin{enumerate}
\item The prime $p$ is absolutely unramified in $F$.
\item For $w$ a place of $F$ with $w|p$, $\pi_w$ is an unramified principal series representation. 
\item Moreover, for each $w|p$ we have $\pi_w \cong \mathrm{Ind}(\mu_1\otimes\mu_2)$ (normalised parabolic  induction from a Borel subgroup) with $\mu_1, \mu_2$ \emph{distinct} unramified characters of $F_w^\times$.
\item The residual representation \[\overline{r}_{p,\iota}(\pi): \Gal(\overline{F}/F)\rightarrow \GL_2(\overline{\F}_p)\] is irreducible.
\end{enumerate}

Then \[WD(r_{p,\iota}(\pi)|_{\Gal(\overline{F}_v/F_v)})^{F\mhyphen ss}\cong\sigma^{\iota}(\pi_v).\]
\end{theorem}
\begin{remark}
We will not discuss the question of local--global compatibility at places dividing $p$ in this paper. However, under the `$p$-distinguished' hypothesis (3) made above, an argument using analytic continuation of crystalline periods (as done in \cite{Jorzathesis} for the case of low weight Siegel modular forms) shows that the representations $r_{p,\iota}(\pi)$ are crystalline at places dividing $p$, with the expected associated Weil--Deligne representation.
\end{remark}

The main novelty in this Theorem is that we are able to establish non-vanishing of the monodromy action for a Galois representation which is only indirectly related (by congruences) to the cohomology of Shimura varieties. 

\subsection{The technical hypotheses}
We will say a few words about the hypotheses in Theorem \ref{mylgc}. Hypotheses (1) and (2) are satisfied for all but finitely many $p$. Hypothesis (1) appears because we use results of Kassaei \cite{kassaei} on analytic continuation and gluing of overconvergent Hilber modular forms. At least for forms of parallel weight, results have been announced by Pilloni, Stroh \cite{PS-bt} and Sasaki \cite{shuartin} which apply without restriction on the ramification of $p$, so it seems reasonable to hope that this hypothesis could be removed in future. Once hypothesis (1) is removed, hypothesis (2) can be substantially relaxed by first making a base change to an extension $F'/F$ in which $p$ ramifies. In fact, recent work of Kassaei, Sasaki and Tian already allows us to relax hypothesis (2) slightly, but we have kept to the 'unramified' setting for expository reasons.

Hypothesis (3) is again present in \cite{kassaei} and has been removed by \cite{PS-bt}, \cite{shuartin} in their setting, but we also make some independent use of this condition.

We assume hypothesis (4) in order to apply the results of \cite{memathann} (it simplifies the analysis of completed cohomology of Shimura curves). It should be possible to avoid this assumption by modifying the arguments of loc. cit. to work with overconvergent cohomology \cite{AshSte}.  

To summarise, we believe that the techniques described in this paper could in principle prove a version of the above theorem with just one technical hypothesis: for each place $w|p$ of $F$, $\pi_w$ is not a twist of the Steinberg representation. To handle the remaining cases seems to require a new idea or a different method.

\subsection{Sketch of proof}
We now sketch the proof of Theorem \ref{mylgc}. By twisting and base change we may assume that $[F:\Q]$ is even, that $\pi_{v}$ is an unramified twist of the Steinberg representation, and that we have an auxiliary finite place $w$ such that $\pi_w$ is also an unramified twist of Steinberg. In the light of Remark \ref{weaklgc} we just need to show that the Weil--Deligne representation associated to the local representation \[\rho_v=r_{p,\iota}(\pi)|_{\Gal(\overline{F}_v/F_v)}\] has a non-zero monodromy operator. Equivalently, we must show that $\rho_v$ is not an unramified representation. We will assume that $\rho_v$ \emph{is} unramified and obtain a contradiction. 

The auxiliary Steinberg place $w$ allows us to find systems of Hecke eigenvalues attached to $\pi$ in the $p$-adically completed cohomology of Shimura curves associated to indefinite quaternion algebras. In this context, a $p$-adic analogue of Mazur's principle (proved in \cite{memathann}) allows us to show that, since the representation $\rho_{v}$ is {unramified}, we can strip $v$ from the level of $\pi$. More precisely we can produce overconvergent Hilbert modular forms, with level prime to $v$, which share the same system of Hecke eigenvalues (outside $v$) as $\pi$. Finally, a generalisation, due to Kassaei, of Buzzard--Taylor's gluing and analytic continuation of overconvergent eigenforms allows us to produce a \emph{classical} Hilbert modular form, with level prime to $v$, contributing to $\pi$. But we assumed that $\pi_v$ was (a twist of) Steinberg, so $\pi$ contains no non-zero $\GL_2(\OO_{F_v})$--invariant vectors. Therefore we obtain the desired contradiction.

Readers familiar with the theory of $p$-adic and overconvergent automorphic forms may find it amusing that we make use of three different avatars of overconvergent automorphic forms in this paper --- firstly geometrically defined overconvergent Hilbert modular forms (sections of automorphic line bundles on strict neighbourhoods of the ordinary locus in Hilbert modular varieties as in, for example, \cite{KL}), secondly the overconvergent automorphic forms on definite quaternion algebras defined by Buzzard \cite{Bu2} and finally the spaces obtained by applying Emerton's locally analytic Jacquet functor to the $p$-adically completed cohomology of Shimura curves (\cite{Emint,memathann}).

The organisation of our paper is as follows: first (\ref{notred}) we reduce the statement of Theorem \ref{mylgc} to a special case. Then (\ref{evars}) we describe the various eigenvarieties we will make use of, together with their construction. We next explain how to find systems of Hecke eigenvalues arising from the automorphic representations $\pi$ of interest in the completed cohomology of Shimura curves, where we apply the level lowering results of \cite{memathann} (\ref{cc}). Finally we return to overconvergent modular forms on Hilbert modular varieties and apply a key result of Kassaei \cite{kassaei} (\ref{end}). 
\subsection{Other approaches}
We note that Luu \cite[\S 3.2]{Luu} has also recently described an approach to proving some cases of local--global compatibility for Hilbert modular forms of low weight, when $\pi$ satisfies an ordinarity hypothesis at places dividing $p$. We will use the same notation as in the previous subsection. Using the assumption that $\rho_v$ is unramified, Luu applies a modularity lifting theorem to produce an ordinary $p$-adic Hilbert modular form $g$ with level prime to $v$ and the same system of Hecke eigenvalues (outside $v$) as $\pi$, and imposes a hypothesis that amounts to ruling out the existence of $g$. It may also be possible to show that $g$ is classical (hence obtaining a contradiction) using a variant of the methods applied in the parallel weight $1$ case (as in, for example, \cite{kassaei}), but it is not obvious to the author that these methods can be easily applied. One obstacle is that, unlike the parallel weight $1$ case, not all the $p$-stabilisations of the newform generating $\pi$ are ordinary. The advantage of our method is that these $p$-stabilisations still have finite slope, so we can apply the level lowering results of \cite{memathann} to obtain several overconvergent forms and then apply Kassaei's work \cite{kassaei}.
\subsection{Examples}
We note that until recently not many examples of Hilbert modular forms of partial weight one were known --- in particular, as far as the author knows, the only examples around were CM forms (and therefore not twists of Steinberg at any place). However, recently Moy and Specter \cite{MoySpe} have computed examples which are non-CM and Steinberg at a finite place. 

\section{Notation and reductions}\label{notred}
We let $F/\Q$ be a totally real, finite extension of $\Q$, and fix a prime $p$ such that $p$ is absolutely unramified in $F$. We denote by $\Sigma$ the set of Archimedean places of $F$. We write $\A_F$ for the adeles of $F$ and $\A_{F}^\infty$ for the finite adeles. For a finite place $v$ of $F$ we write $\OO_v$ for the ring of integers in the local field $F_v$. For convenience, we fix once and for all a choice of uniformiser $\varpi_v \in \OO_v$ for each finite place $v$. 

We denote by $\Sigma_p$ the set of embeddings $F \hookrightarrow \overline{\Q}_p$ and for a place $\pf|p$ denote by $\Sigma_\pf$ those embeddings which extend to $F_\pf$. Our fixed isomorphism $\iota: \overline{\Q}_p \cong \C$ induces a bijection between $\Sigma$ and $\Sigma_p$.

We will need some notation for (limit of) discrete series representations of $\GL_2(\R)$. For $k \ge 2$ and $w$ integers of the same parity we let $D_{k,w}$ denote the discrete series representation of $\GL_2(\R)$ with central character $t \mapsto t^{-w}$ defined in section 0.2 of \cite{Car}. For $k=1$ and $w$ an odd integer we define $D_{1,w}$ to be the limit of discrete series $\mathrm{Ind}(\mu,\nu)$ where the induction is a normalised parabolic induction from a Borel subgroup and $\mu, \nu$ are characters of $\R^\times$ defined by $\mu(t)=|t|^{-w/2}\mathrm{sgn}(t)$ and $\nu(t)=|t|^{-w/2}$.

We assume $\pi$ is an automorphic representation of $\GL_2(\A_F)$ as in the statement of Theorem \ref{mylgc}, and that $v$ is a finite place of $F$ with $\pi_v$ special. First we note that by twisting and making a quadratic base change (to an extension of $F$ in which $v$ splits) we can assume without loss of generality that $[F:\Q]$ is even, that $\pi_v$ is an unramified twist of the Steinberg representation and that there is another finite place $v'$ with $\pi_{v'}$ also an unramified twist of Steinberg.

With this in mind, for the rest of the paper we fix a cuspidal automorphic representation $\pi_0$ of $\GL_2(\A_F)$ as in the statement of Theorem \ref{mylgc} and moreover assume that
\begin{itemize}
\item $[F:\Q]$ is even,
\item $\pi_{0,v}$ is an unramified twist of Steinberg,
\item there is a finite place $v'$, coprime to $p$ and $v$ such that $\pi_{0,v'}$ is an unramified twist of Steinberg,
\item $\pi_{0,\tau} \cong D_{k_{\tau},w}$ for each $\tau \in \Sigma$ with $k_\tau$ an odd positive integer and $w$ an odd integer (independent of $\tau$).
\end{itemize}

\subsection{Hilbert modular forms and varieties}\label{Hilbsection}
We now proceed to reduce Theorem \ref{mylgc} to a statement about Hilbert modular forms. Let $L\subset\overline{Q}_p$ be a finite extension of $\Q_p$ which contains the image of every embedding from $F$ into $\overline{\Q}_p$. Let $\n$ and $\m$ be two coprime ideals of $\OO_F$, both coprime to $p$, and suppose that $\n$ is divisible by a rational integer which is $\ge 4$. 
\begin{definition}\label{Hilb}
Let $\mathfrak{c}$ be a fractional ideal of $F$ such that $\mathbf{N}\mathfrak{c}$ is coprime to $\m\n$. We denote by $\underline{\mathscr{M}}^{\mathfrak{c}}(\n,\m)$ the functor from schemes over $\OO_L$ to sets taking $S$ to isomorphism classes of tuples $(A,\lambda,[\eta])$, where\begin{itemize}
\item $A$ is a Hilbert--Blumenthal Abelian scheme over $S$ (in particular $A$ is equipped with an action of $\OO_F$, \cite[1.1]{KL})
\item $\lambda$ is a $\mathfrak{c}$-polarisation of $A$ (\cite[1.2]{KL})
\item $\eta$ is an $\OO_F$-equivariant closed immersion of $S$-group schemes \[\eta: (\OO_F/\m\n)^\vee \hookrightarrow A[\m\n]\] and $[\eta]$ is its equivalence class under the natural action of $(\OO_F/\m)^\times$. Here $(\OO_F/\m\n)^\vee$ denotes the Cartier dual of the constant $S$-group scheme (with $\OO_F$ action) $\OO_F/\m\n$.
\end{itemize}
\end{definition}

Under our assumptions, the functor $\underline{\mathscr{M}}^{\mathfrak{c}}(\n,\m)$ is representable by a smooth $\OO_L$-scheme ${\mathscr{M}}^{\mathfrak{c}}(\n,\m)$, and there is a good theory of toroidal and minimal compactifications (see \cite[1.6, 1.8]{KL} --- although we use a slightly modified level structure, everything in this reference goes through).

Having fixed a choice of representatives $\mathfrak{c}_1,...,\mathfrak{c}_h$ for the narrow class group of $F$ (which we may assume all satisfy $\mathbf{N}\mathfrak{c}_i$ coprime to $p\n\m$), we define $\mathscr{M}(\n,\m)$ to be the disjoint union of the $\mathscr{M} ^{\mathfrak{c}_i}(\n,\m)$.

For $\kappa$ a homomorphism \[\kappa: \mathrm{Res}^{\OO_F}_\Z \G_{m/\OO_L} \rightarrow \G_{m/\OO_L}\]  we have line bundles $\underline{\omega}^\kappa$ as in \cite[1.4.2]{KL} on $\mathscr{M}(\n,\m)$, which extend to suitable toroidal compactifications. We assume that $w$ is an integer such that the character $\mathbf{N}^w\cdot\kappa^{-1}$ admits a square root. Then, following \cite[1.11]{KL} we define Hecke operators $T_\af$ for each prime ideal $\af$ of $\OO_F$ coprime to $\m\n$ and $U_\af$ for prime ideal $\af$ dividing $\m\n$. 

Denote by $\T(\n,\m)$ the polynomial algebra over $\Z$ generated by symbols $T_\af$ and $U_\af$ as above. This algebra naturally acts on $H^0(\mathscr{M}(\n,\m)_L,\underline{\omega}^{\kappa})$, and we say that a non-zero element $f$ of this $L$-vector space is a \emph{Hecke eigenform} if the action of $\T(\n,\m)$ preserves the one-dimensional space $L\cdot f$. If $f$ is a cuspidal Hecke eigenform we denote by $\theta_f$ the character $\T(\n,\m)\rightarrow L$ giving the action of the Hecke algebra on $f$. After projecting to a classical Hilbert modular form (as described in \cite[1.11.8]{KL}) $f$ generates (a model over $L$ for) the finite part $\pi_f^{\infty}$ of a cuspidal automorphic representation $\pi_f$ of $\GL_2(\A_F)$. For $\tau \in \Sigma$ we have $\pi_{\tau} = D_{k_\tau,w}$, where the $k_\tau$ can be read off from $\kappa$ (using the bijection between $\Sigma$ and $\Sigma_p$). The central character of $\pi_f$ is an algebraic Hecke character of the form $\chi_f\omega^{-w}$, where $\omega$ is the norm character and $\chi_f$ is a finite order character. The Galois representation $r_{\pi_f,\iota}$ is the unique semisimple representation, unramified outside $p\m\n$, such that for $\q\nmid p\m\n$ a geometric Frobenius element at $\q$ has characteristic polynomial $X^2-\theta_f(T_\q)X + (\mathbf{N}\q)^{1-w}\iota^{-1}\chi_f(\varpi_\q)$

In the subsequent part of the paper we will prove:
\begin{proposition}\label{proplgc}
Suppose $\m = \q\lf$ is a product of two prime ideals. Let $f \in H^0(\mathscr{M}(\n,\m)_L,\underline{\omega}^{\kappa})$ be a Hecke eigenform. Suppose that:
\begin{itemize}
\item The local factor $\pi_{f,\q}$ is an unramified twist of the Steinberg representation.
\item The residual representation \[\overline{r}_{p,\iota}(\pi_f): \Gal(\overline{F}/F)\rightarrow \GL_2(\overline{\F}_p)\] is irreducible.
\item The polynomials $X^2-\theta_f(T_\pf)X + (\mathbf{N}\pf)^{1-w}\iota^{-1}\chi_f(\varpi_\pf)$ have distinct roots for all prime ideals $\pf|p$.
\item The local representation \[{r}_{p,\iota}(\pi_f)|_{\Gal(\overline{F}_\lf/F_\lf)}\] is unramified.
\end{itemize}

Then there exists a Hecke eigenform $g \in H^0(\mathscr{M}(\n,\q)_L,\underline{\omega}^{\kappa})$ such that the Hecke eigenvalues of $f$ and $g$ coincide outside $\lf$ and $\theta_f(U_\lf)$ is one of the roots of $X^2-\theta_g(T_\lf)X + (\mathbf{N}\lf)^{1-w}\iota^{-1}\chi_f(\varpi_\lf)$.
\end{proposition}

We now explain why this proposition is sufficient to deduce Theorem \ref{mylgc}. 
\begin{lemma}
Proposition \ref{proplgc} implies Theorem \ref{mylgc}.
\end{lemma}
\begin{proof}
Recall that we have reduced the statement of Theorem \ref{mylgc} to the case of $\pi_0$ as described immediately before Section \ref{Hilbsection}. We denote the finite part of $\pi_0$ by $\pi_{0}^\infty$. For $\m$ any non-zero ideal of $\OO_F$ we define congruence subgroups $U_0(\m)$ and $U_1(\m)$ to be the subgroups of $\GL_2(\widehat{\Z})$ given by:
\begin{eqnarray*}
U_0(\m) &= &\{g \in \GL_2(\widehat{\Z}): g = \begin{pmatrix}
* & *\\0&*
\end{pmatrix} \hbox{ mod }\m \}\\
U_1(\m) &= &\{g \in \GL_2(\widehat{\Z}): g = \begin{pmatrix}
* & *\\0&1
\end{pmatrix} \hbox{ mod }\m \}
\end{eqnarray*}

Now we denote by $\lf$ and $\q$ the prime ideals of $\OO_F$ corresponding to the places $v$ and $v'$ respectively, and suppose that $\n$ is an ideal of $\OO_F$, coprime to $p\q\lf$ and divisible by a rational integer $\ge 4$, such that $(\pi_{0}^\infty)^{U_1(\n)\cap U_0(\q\lf)}\ne 0$. The isomorphism $\iota: \overline{\Q}_p \rightarrow \C$ induces a bijection between $\Sigma$ and the embeddings $\Sigma_p$ from $F \hookrightarrow L$, so using this we can associate a character \[\kappa: \mathrm{Res}^{\OO_F}_\Z \G_{m/\OO_L} \rightarrow \G_{m/\OO_L}\] to the $\Sigma$-tuple of integers $\{k_\tau\}_{\tau \in \Sigma}$ describing the Archimedean part of $\pi_0$. Now (possibly enlarging $L$) there is a Hecke eigenform $f \in H^0(\mathscr{M}(\n,\q\lf)_L,\underline{\omega}^{\kappa})$ with $\pi_f \cong \pi_0$, and so $f$ satisfies the hypotheses of Proposition $\ref{proplgc}$. This Proposition gives us a Hecke eigenform $g \in H^0(\mathscr{M}(\n,\q)_L,\underline{\omega}^{\kappa})$ with $\pi_g \cong \pi_0$, and so the space $(\pi_{0}^\infty)^{U_1(\n)\cap U_0(\q)}$ is non-zero, which contradicts the assumption that $\pi_{0,\lf}$ is a twist of Steinberg.
\end{proof}

We end this section by discussing Hilbert modular varieties with Iwahori level at $p$ and a definition of overconvergent Hilbert modular forms.

\begin{definition}
Let $\mathfrak{c}$ be a fractional ideal of $F$ such that $\mathbf{N}\mathfrak{c}$ is coprime to $p\m\n$. We denote by $\underline{\mathscr{M}}^{\mathfrak{c},\mathrm{Iw}}(\n,\m)$ the functor from schemes over $\OO_L$ to sets taking $S$ to isomorphism classes of tuples $(A,\lambda,[\eta],H)$ up to isomorphism, where $(A,\lambda,[\eta])$ are as in Definition \ref{Hilb} and $H$ is a finite flat subgroup scheme of $A[p]$, stable under the action of $\OO_F$, of rank $p^{[F:\Q]}$, and isotropic with respect to the $\lambda$-Weil pairing.
\end{definition}
The functor $\underline{\mathscr{M}}^{\mathfrak{c},\mathrm{Iw}}(\n,\m)$ is represented by an $\OO_L$-scheme ${\mathscr{M}}^{\mathfrak{c},\mathrm{Iw}}(\n,\m)$. We denote by ${\mathscr{M}}^{\mathrm{Iw}}(\n,\m)$ the disjoint union of these over suitable representatives of the narrow class group, as before.  

Denoting the rigid generic fibre of ${\mathscr{M}}^{\mathrm{Iw}}(\n,\m)$ by $M^{\mathrm{Iw}}(\n,\m)_L$ we can consider sections of the line bundles $\underline{\omega}^{\kappa}$ over strict neighbourhoods of the locus in $M^{\mathrm{Iw}}(\n,\m)_L$ where $H$ is Cartier dual to $\OO_F/p$ (i.e. the multiplicative ordinary locus), to obtain a space $M_\kappa^\dagger(\n,\m)$ of overconvergent modular forms of weight $\kappa$ (see \cite[\S 4]{kassaei}). 
%For $\pf|p$ we define an operator $U_\pf^\dagger$ on $M_\kappa^\dagger(\n,\m)$ as in \cite[Definition 5.1]{kassaei}. 

\section{Eigenvarieties}\label{evars}
%revised
We will need to make use of eigenvarieties constructed in different contexts. To clarify the relationship between these eigenvarieties, we are going to follow the abstract approach of \cite{BC}. First we need to discuss the weight spaces over which our eigenvarieties will live.

\subsection{Weight spaces}
We denote by $G = \prod_{\pf|p}G_\pf = \mathrm{Res}_{F/\Q}(\GL_2)(\Q_p)$, where $G_\pf = \GL_2(F_\pf)$. We denote by $B = \prod_{\pf|p}B_\pf \subset G$ the Borel subgroup comprising upper triangular matrices, and denote by $T = \prod_{\pf|p}T_\pf$ the maximal torus comprising diagonal matrices. We denote by $N_\pf$ the subgroup of $B_\pf$ whose elements have $1$s on the diagonal. Finally, $T_0 = \prod_{\pf|p}T_{0,\pf}\subset T$ is the compact subgroup whose elements have integral entries.

Fix a finite extension $L\subset \overline{\Q}_p$ of $\Q_p$, which we assume contains a normal closure of $F$. The functor taking an $L$-affinoid $\Sp A$ to the set of continuous $A^\times$ valued characters of $T_0 = (\OO_F\otimes_\Z\Z_p)^\times \times (\OO_F\otimes_\Z\Z_p)^\times$ is representable by a rigid analytic space $\widehat{T}_0$ over $L$ \cite[\S2]{Bu1}. Likewise, we have a rigid space $\widehat{T}$ representing continuous characters of $T$.

\begin{definition}
Denote by $\mathscr{W}$ the subspace of $\widehat{T}_0$ whose points correspond to continuous characters of $T_0$ which are trivial on a finite index subgroup of $\OO_F^\times$ (embedded diagonally in $T_0$).
\end{definition}

Suppose we have an algebraic character $\kappa: \mathrm{Res}^{\OO_F}_{\Z}\G_m/\OO_L \rightarrow \G_m/\OO_L$ such that $\mathbf{N}^w\cdot\kappa^{-1}$ admits a square root admits a square root. Recall that $\Sigma_p$ denotes the set of embeddings from $F$ to $L$, so $\kappa$ corresponds to a $\Sigma$-tuple of integers $\{k_{\tau}\}_{\tau\in\Sigma}$, with the same parity as $w$.

The character \[(t_1,t_2) \mapsto \prod_{\tau\in\Sigma_p}\tau(t_1)^{\frac{k_\tau-2-w}{2}}\tau(t_2)^{\frac{-w-k_\tau+2}{2}}\] is a point in $\mathscr{W}(L)$, which we also denote by $\kappa$. 

\begin{definition}
Denote by $\mathscr{W}_\kappa$ the subspace of $\mathscr{W}$ such that maps from an $L$-affinoid $\Sp A$ to $\mathscr{W}_\kappa$ correspond to characters $(\chi_1,\chi_2): (\OO_F\otimes_\Z\Z_p)^\times \times (\OO_F\otimes_\Z\Z_p)^\times \rightarrow A^\times$ such that $\chi_1\chi_2 = \Nm^{-w}$ and $\chi_1\chi_2^{-1} =  \kappa\Nm^{-2}\cdot(\tau\circ \Nm)$, where $\tau$ is a continuous character $\tau:\Z_p^\times\rightarrow A^\times$ with \[v_p(1-\tau(a)) > \frac{1}{p-1}\] for all $a \in \Z_p^\times$. Here $\Nm$ is the natural extension of the norm map to a continuous map $(\OO_F\otimes_\Z\Z_p)^\times\rightarrow \Z_p^\times$.
\end{definition}

Taking $\tau$ to be the trivial character, we see that $\kappa$ is a point of $\mathscr{W}_\kappa$. In fact, $\mathscr{W}_\kappa$ is a small (one-dimensional) open disc in $\mathscr{W}$, with centre $\kappa$.

\begin{remark}
The weight space $\mathscr{W}$ is isomorphic to the weight space (also denoted by $\mathscr{W}$) defined by Buzzard in \cite[\S 8]{Bu2}. Our character $(\chi_1,\chi_2)$ of $(\OO_F\otimes_\Z\Z_p)^\times \times (\OO_F\otimes_\Z\Z_p)^\times$ corresponds to the character $(n,v) = (\chi_1\chi_2^{-1},\chi_2)$ in Buzzard's weight space. 

The weight space $\mathscr{W}_{\kappa}$ is isomorphic to the weight space (also denoted by $\mathscr{W}_\kappa$) defined by Kisin and Lai in \cite[\S 4.5]{KL}
\end{remark}

We are going to build eigenvarieties interpolating classical Hilbert modular forms over the weight space $\mathscr{W}_\kappa$. These eigenvarieties will be constructed using three different notions of `overconvergent automorphic form'. 
\subsection{Refinements}
For $\pf|p$, denote by $I_\pf$ the Iwahori subgroup of $G_\pf$, comprising matrices which reduce to upper triangular matrices mod $\varpi_\pf$. We write $I$ for the product $\prod_{\pf|p}I_\pf \subset G$.

\begin{definition}
Let $\pi_\pf$ be an irreducible smooth representation of $G_\pf$ on a complex vector space which is either irreducible principal series or a twist of the Steinberg representation. An \emph{accessible refinement} of $\pi_\pf$ is a character $\chi$ of $T_\pf$ such that there is a $G_\pf$-equivariant embedding \[\pi_\pf \hookrightarrow \mathrm{Ind}_{B_\pf}^{G_\pf}\chi.\] 
\end{definition}

\begin{remark}\label{examples} We have the following explicit description of accessible refinements: we have $\pi_\pf \cong \mathrm{Ind}_{B_\pf}^{G_\pf}\mu_1\otimes\mu_2$ or $\pi_\pf \cong \mathrm{St}\otimes\mu$, where $\mu, \mu_i$ are characters of $F_{\pf}^\times$.
\begin{itemize}
\item Suppose $\pi_\pf \cong \mathrm{Ind}_{B_\pf}^{G_\pf}\mu_1\otimes\mu_2$. Then the accessible refinements of $\pi_pf$ are the characters $\mu_1\otimes\mu_2$ and $\mu_2\otimes\mu_1$. Note that if $\mu_1 = \mu_2$ these refinements are the same.
\item Suppose $\pi_\pf \cong \mathrm{St}\otimes\mu$. Then $\pi_\pf$ is isomorphic to the unique irreducible subrepresentation of the normalised induction $\mathrm{Ind}_{B_\pf}^{G_\pf}\mu|\cdot|_\pf^{1/2}\otimes\mu|\cdot|_\pf^{-1/2}$ and the unique accessible refinement of $\pi_\pf$ is $\mu|\cdot|_\pf^{1/2}\otimes\mu|\cdot|_\pf^{-1/2}$.
\end{itemize}
\end{remark}
For an ideal $\n$ of $\OO_F$, coprime to $p$, we denote by $\cH(\n)$ the free commutative polynomial ring over $\Z$ on generators labelled $T_v$ and $S_v$ for places $v$ of $F$ not dividing $p\n$ and $U_v$ for places $v|\n$.

For $\m,\n$ ideals of $\OO_F$ (coprime to each other and to $p$), suppose $\pi$ is a cuspidal automorphic representation of $\GL_2(\A_F)$ with $(\pi^{\infty p})^{U_1(\n)\cap U_0(\m)} \ne 0$ for some $n$ and $\pi_{\tau}=D_{k_\tau,w}$ for each $\tau\in \Sigma$, with $k_\tau \ge 2$ and $k_\tau = w$ mod $2$ for each $\tau$. We say that such a $\pi$ is a classical automorphic representation of tame level $U_1(\n)\cap U_0(\m)$ and weight $(k,w)$. We moreover say that $\pi$ has finite slope if the smooth Jacquet modules $J_{B_\pf}(\pi_\pf)$ are non-zero for all $\pf|p$. Equivalently, the $\pi_\pf$ are all irreducible principal series or twists of Steinberg (since local factors of a cuspidal automorphic representation are generic). 

\begin{definition}
Suppose $\pi$ is a finite slope representation. An \emph{accessible refinement} of $\pi$ is a character \[\chi = \bigotimes_{\pf|p}\chi_{\pf}:T \rightarrow  \C^{\times}\] such that each $\chi_\pf$ is an accessible refinement of $\pi_\pf$.
\end{definition}
\begin{remark}
The possible accessible refinements of finite slope automorphic representations are completely classified by Remark \ref{examples}.
\end{remark}
\begin{definition}
We say that a finite slope representation is \emph{unramified} if $\pi_\pf$ is either an unramified principal series or unramified twist of Steinberg for all $\pf|p$. 
\end{definition}
\begin{remark}
The following two conditions are easily seen to be equivalent to $\pi$ being unramified:
\begin{itemize}
\item For every $\pf|p$, $\pi_\pf$ has non-zero invariants for the Iwahori subgroup $I_\pf$.
\item Every accessible refinement of $\pi$ factors through $T/T_0$.
\end{itemize}
\end{remark}
\begin{definition}
Suppose $\pi$ is a finite slope representation of weight $(k,w)$, and that $\chi$ is an accessible refinement of $\pi$. We define a continuous character \[\nu(\pi,\chi): T \rightarrow \overline{\Q}_p^\times\] by \[\nu(\pi,\chi) = \bigotimes_{\pf|p} \iota^{-1}\chi_\pf|\cdot|_\pf^{1/2} \otimes |\cdot|_\pf^{-1/2}\prod_{\tau\in \Sigma_\pf}\tau^{(k_\tau-2-w)/2}\otimes\tau^{(-w-k_\tau+2)/2}.\]
\end{definition}
\begin{remark}
Our discussion of accessible refinements and the definition of the character $\nu(\pi,\chi)$ is parallel to that of \cite[\S 1.4]{ChenFamGal}. We will see in section \ref{emerton} that these definitions are completely natural when constructing eigenvarieties using completed cohomology and Emerton's locally analytic Jacquet functor.
\end{remark}

\subsection{Abstract eigenvarieties}
\begin{definition}
Let $\mathscr{W}'$ be any subspace of $\mathscr{W}$ which is an admissible increasing union of open affinoids (in practice it will be $\mathscr{W_\kappa}$). Let $\mathscr{H}$ be a commutative $\Z$-algebra and let $\mathscr{Z}$ be a subset of $\Hom(\mathscr{H},\overline{\Q}_p) \times \widehat{T}(\overline{\Q}_p)$ whose image in $\widehat{T}_0(\overline{\Q}_p)$ is an accumulation\footnote{A set is accumulation if each point $z \in Z$ has a basis of affinoid neighbourhoods $U$ such that $Z\cap U$ is Zariski dense in $U$.} and Zariski dense subset of $\mathscr{W}'$. Denote by $Y$ the fibre product of $\widehat{T}$ and $\mathscr{W}'$ over $\widehat{T}_0$. Then an \emph{eigenvariety} for the triple $(\mathscr{H},\mathscr{Z},\mathscr{W}')$ is a reduced rigid space $X$ over $L$ equipped with
\begin{itemize}
\item A ring homomorphism $\psi: \mathscr{H}\rightarrow\OO(X)$
\item A finite morphism $\nu: X \rightarrow Y$
\item An accumulation and Zariski dense subset $Z\subset X(\overline{\Q}_p)$ (which we refer to as the `classical subset')
\end{itemize}
such that the following are satisfied
\begin{enumerate}
\item For all open affinoids $V \subset Y$ the natural map
\[\psi\otimes\nu^*:\mathscr{H}\otimes\OO(V)\rightarrow\OO(\nu^{-1}(V))\] is surjective.
\item The natural evaluation map \begin{eqnarray*} X(\overline{\Q}_p) &\rightarrow &\Hom(\mathscr{H},\overline{\Q}_p)\\
x &\mapsto & \psi_x:=(h\mapsto \psi(h)(x))
\end{eqnarray*}
induces a bijection $z \mapsto (\psi_z,\nu(z))$ from $Z$ to $\mathscr{Z}$.
\end{enumerate}
\end{definition}
The key property of the above definition is that an eigenvariety is unique up to unique isomorphism, by \cite[Proposition 7.2.8]{BC}.
\begin{remark}\label{remgalalt}
An alternative way to abstractly characterise eigenvarieties in our context is as Zariski closures of sets of classical points in the rigid space given by the product of the rigid generic fibre of a Galois (pseudo)~deformation ring and some affine spaces or copies of $\G_m$ (to keep track of additional Hecke eigenvalues).
\end{remark}

For $\m,\n$ ideals of $\OO_F$ (coprime to each other and to $p$), suppose $\pi$ is an unramified representation of tame level $U_1(\n)\cap U_0(\m)$.

There is a natural action of $\cH(\m\n)$ on $(\pi^\infty)^{U_1(\n)\cap U_0(p\m)}$, where we let $T_v$ and $U_v$ act by double coset operators \[[U\begin{pmatrix}\varpi_v & 0\\0 & 1 \end{pmatrix}U]\] and let $S_v$ act by the double coset operator  \[[U\begin{pmatrix}\varpi_v & 0\\0 & \varpi_v \end{pmatrix}U]\] where $U = U_1(\n)\cap U_0(p\m)$.

Given such a $\pi$, we obtain a subset $\mathscr{Z}(\n,\m)_{\pi}$ of $\Hom(\cH(\m\n),\overline{\Q}_p)\times \widehat{T}(\overline{\Q}_p)$ by taking pairs \[(\psi,\nu(\pi,\chi))\] where $\psi$ is a character corresponding (via $\iota$) to a Hecke eigenform in $(\pi^\infty)^{U_1(\n)\cap U_0(p\m)}$ and $\chi$ is an accessible refinement of $\pi$. Note that any element of the set $\mathscr{Z}(\n,\m)_{\pi}$ determines $\pi$, by strong multiplicity one. The choice of accessible refinement $\chi$ corresponds to a choice of $U_\pf$-eigenvalue in the space $(\pi_\pf)^{I_\pf}$ for each $\pf|p$. The character $\kappa_\pi:= \nu(\pi,\chi)|_{T_0}$ is independent of the refinement $\chi$.

\begin{definition}
Let $\m$ and $\n$ be a pair of coprime ideals in $\OO_F$, both coprime to $p$. 
\begin{enumerate}
\item Denote by $\mathscr{Z}(\n,\m)$ the union of the $\mathscr{Z}(\n,\m)_{\pi}$ obtained from unramified $\pi$ with tame level $U_1(\n)\cap U_0(\m)$ such that $\kappa_\pi \in \mathscr{W}_\kappa$.

\item For $\q$ a prime divisor of $\m$, we write $\mathscr{Z}(\n,\m)^{\q\mhyphen sp}$ for the subset of $\mathscr{Z}(\n,\m)$ arising from $\pi$ with $\pi^{U_1(\n)\cap U_0(p\m/\q)} = 0$ (equivalently, the local factor $\pi_{\q}$ is an unramified twist of the Steinberg representation).

\item Similarly, we write $\mathscr{Z}(\n,\m)^{\q\mhyphen ps}$ for the subset of $\mathscr{Z}(\n,\m)$ arising from $\pi$ with $\pi^{U_1(\n)\cap U_0(p\m/\q)} \ne 0$ (equivalently, the local factor $\pi_{\q}$ is an unramified principal series representation).

\item Now we fix the irreducible mod $p$ Galois representation $\rhobar = \overline{r}_{p,\iota}(\pi_f).$ Then we can define $\mathscr{Z}(\n,\m)_{\rhobar}$, $\mathscr{Z}(\n,\m)^{\q\mhyphen sp}_{\rhobar}$ and $\mathscr{Z}(\n,\m)^{\q\mhyphen ps}_{\rhobar}$ to be the subsets of $\mathscr{Z}(\n,\m)$ etc. arising from those $\pi$ with residual Galois representation $\overline{r}_{p,\iota}(\pi)$ isomorphic to $\rhobar$. 
\end{enumerate}
\end{definition} 

\subsection{Buzzard's eigenvarieties}

\begin{theorem}\label{buzeig}
There exist eigenvarieties $\mathscr{E}(\n,\m)_{\rhobar}$, $\mathscr{E}(\n,\m)^{\q\mhyphen sp}_{\rhobar}$ and $\mathscr{E}(\n,\m)^{\q\mhyphen ps}_{\rhobar}$ for the triples \[(\cH(\m\n),\mathscr{Z}(\n,\m)_{\rhobar},\mathscr{W}_\kappa), (\cH(\m\n),\mathscr{Z}(\n,\m)^{\q\mhyphen sp}_{\rhobar},\mathscr{W}_\kappa), (\cH(\m\n),\mathscr{Z}(\n,\m)^{\q\mhyphen ps}_{\rhobar},\mathscr{W}_\kappa).\] 

We denote the classical subsets of these eigenvarieties by $Z(\n,\m)_{\rhobar}$, ${Z}(\n,\m)^{\q\mhyphen sp}_{\rhobar}$ and ${Z}(\n,\m)^{\q\mhyphen ps}_{\rhobar}$. The following properties are satisfied by these eigenvarieties:
\begin{itemize}
\item There are closed immersions \[\mathscr{E}(\n,\m)^{\q\mhyphen sp}_{\rhobar}\hookrightarrow\mathscr{E}(\n,\m)_{\rhobar}\] and \[\mathscr{E}(\n,\m)^{\q\mhyphen ps}_{\rhobar}\hookrightarrow\mathscr{E}(\n,\m)_{\rhobar}\] commuting with the maps to weight space and respecting the homomorphisms $\psi$ from $\cH(\m\n)$, with images (respectively $X$, $Y$) given by unions of irreducible components (in the sense of \cite{con}). 
\item Each irreducible component of $\mathscr{E}(\n,\m)_{\rhobar}$ is contained in precisely one of $X$ and $Y$. 
\item We have $Z(\n,\m)_{\rhobar}\cap X = {Z}(\n,\m)^{\q\mhyphen sp}_{\rhobar}$ and $Z(\n,\m)_{\rhobar}\cap Y = {Z}(\n,\m)^{\q\mhyphen ps}_{\rhobar}.$
\item There is a map \[\mathscr{E}(\n,\m)^{\q\mhyphen ps}_{\rhobar} \rightarrow \mathscr{E}(\n,\m/\q)_{\rhobar},\] surjective on closed points, for which the pre-image of a closed point $x \in \mathscr{E}(\n,\m/\q)_{\rhobar}$ is indexed by the roots of the Hecke polynomial $X^2-\psi(T_\q)(x)X + \mathbf{N}{\q}\psi(S_\q)(x)$.
\end{itemize}
\end{theorem}
\begin{proof}We fix a definite quaternion algebra $D/F$, ramified precisely at the infinite places of $F$ and an isomorphism $(D\otimes_F \A_F^\infty)^\times \cong \GL_2(\A_F^\infty)$. Buzzard's definition \cite[Part III]{Bu2} of overconvergent automorphic forms on $D$,  with tame level $U_1(\n)\cap U_0(\m)$ allows us to construct an eigenvariety $\mathscr{E}(\n,\m)$ for the triple $(\cH(\m\n),\mathscr{Z}(\n,\m),\mathscr{W}_\kappa)$. The Zariski density and accumulation properties for the classical points follows from a special case of the classicality criterion given by \cite[Theorem 3.9.6]{L}. To obtain the eigenvariety for $(\cH(\m\n),\mathscr{Z}(\n,\m)_{\rhobar},\mathscr{W}_\kappa)$ we just take the union of connected components in $\mathscr{E}(\n,\m)$ whose closed points have associated residual Galois representation isomorphic to $\rhobar$.

We then define $\mathscr{E}(\n,\m)^{\q\mhyphen sp}_{\rhobar}$ to be the Zariski closure in $\mathscr{E}(\n,\m)_{\rhobar}$ of the subset $Z(\n\m)_{\rhobar}^{\q\mhyphen sp}\subset Z(\n\m)_{\rhobar}$ corresponding to systems of Hecke eigenvalues in $\mathscr{Z}(\n,\m)_{\rhobar}^{\q\mhyphen sp}$. Similarly, we define $\mathscr{E}(\n,\m)^{\q\mhyphen ps}_{\rhobar}$ to be the Zariski closure in $\mathscr{E}(\n,\m)_{\rhobar}$ of the subset $Z(\n\m)_{\rhobar}^{\q\mhyphen ps}$. We now need to check that $Z(\n\m)_{\rhobar}^{\q\mhyphen sp}$ and $Z(\n\m)_{\rhobar}^{\q\mhyphen ps}$ are accumulation subsets in their Zariski closures, together with the rest of the assertions of the Theorem. 

We can deduce everything we need by applying the results of \cite[7.8]{BC} on the family of Weil--Deligne representations carried by an eigenvariety (see also \cite{Paulinlg}). We proceed as follows: denote by $X \subset\mathscr{E}(\n,\m)_{\rhobar}$ the reduced closed subspace given by the union of irreducible components where the monodromy operator in the associated family of Weil--Deligne representations is generically non--zero  --- we call such a component `generically special'. More precisely, in the notation of \cite[7.8]{BC} a generically special component $W$ is one which for all closed points $x \in W$ we have $N_{s(x)}^{\mathrm{gen}}$ non-zero for $s(x)$ any germ of an irreducible component at $x$ which is contained in $W$. Then we claim that \[X\cap Z(\n\m)_{\rhobar} = Z(\n\m)_{\rhobar}^{\q\mhyphen sp}.\] Indeed, if $x \in Z(\n\m)_{\rhobar}^{\q\mhyphen sp}$ then, by \cite[Proposition 7.8.19 (iii)]{BC} and local--global compatibility at $\q$ for the automorphic representation $\pi_x$ \cite{Car}, every irreducible component passing through $x$ is generically special. 

Conversely, if $x \in X\cap Z(\n\m)_{\rhobar}$ then the Weil--Deligne representation at $\q$ associated to $\rho_x$ is forced to have the form $(W\oplus W(1),N)$, where $W$ is a one--dimensional $\overline{k(x)}$-vector space with an unramified action of the Weil group $W_\q$ and $W(1)$ denotes the twist of $W$ by the $p$-adic cyclotomic character ($N$ could be zero or non--zero). This means that either $x \in Z(\n\m)_{\rhobar}^{\q\mhyphen sp}$ or the local factor $\pi_\q$ of the automorphic representation $\pi$ associated to $x$ is one--dimensional. The latter situation cannot occur, since $\pi$ is a cuspidal automorphic representation of $\GL_2(\A_F)$ and therefore its local factors are infinite--dimensional (this follows from the existence of a global Whittaker model, for example see the proof of Theorem 11.1, \cite{JL}). 

Now it is easy to deduce the accumulation property for $Z(\n\m)_{\rhobar}^{\q\mhyphen sp}$ from the accumulation property for $Z(\n\m)_{\rhobar}$.

Similar arguments apply if we take $Y \subset\mathscr{E}(\n,\m)_{\rhobar}$ to be the reduced closed subspace given by the union of irreducible components where the monodromy operator in the associated family of Weil--Deligne representations is generically zero  --- we call such a component `generically principal series'.

Finally we need to construct the map \[\mathscr{E}(\n,\m)^{\q\mhyphen ps}_{\rhobar} \rightarrow \mathscr{E}(\n,\m/\q)_{\rhobar}.\] This can be obtained by giving an alternate construction of $\mathscr{E}(\n,\m)^{\q\mhyphen ps}_{\rhobar}$ --- indeed, by the uniqueness of abstract eigenvarieties, $\mathscr{E}(\n,\m)^{\q\mhyphen ps}_{\rhobar}$ coincides with the nilreduction of the covering of $\mathscr{E}(\n,\m/\q)_{\rhobar}$ given by the roots of $X^2-\psi(T_\q)X + \mathbf{N}{\q}\psi(S_\q)(x)$ (this has a natural interpretation as a relative spectrum over $\mathscr{E}(\n,\m/\q)_{\rhobar}$). 
\end{proof}
\subsection{Kisin and Lai's eigenvarieties}\label{KLeig}
Now we let $f$ be a Hecke eigenform as in the statement of Proposition \ref{proplgc}. For each prime ideal $\pf|p$ we denote the two distinct roots of $X^2-\theta_f(T_\pf)X + (\mathbf{N}\pf)^{1-w}\iota^{-1}\chi_f(\varpi_\pf)$ by $\alpha_\pf$ and $\beta_\pf$. Then for each subset  $S \subset \{\pf|p\}$ there is a unique Hecke eigenform \[f_S \in H^0({\mathscr{M}}^{\mathrm{Iw}}(\n,\q\lf)_L,\underline{\omega}^\kappa)\] whose Hecke eigenvalues away from $p$ are the same as $f$ and for $\pf|p$ we have $\theta_{f_S}(U_\pf) = \alpha_{\pf}$ if $\pf \in S$ and $\theta_{f_S}(U_\pf) = \beta_{\pf}$ if $\pf \notin S$.

We moreover define a point $\nu_S \in \widehat{T}(\overline{\Q}_p)$ to be given by the character \[\bigotimes_{\pf|p} \chi_{\pf,1}\otimes\chi_{\pf,2}\prod_{\tau\in \Sigma_\pf}\tau^{(k_\tau-2-w)/2}\otimes\tau^{(-w-k_\tau+2)/2}\] where the $\chi_{\pf,i}$ are characters of $F_\pf^\times/\OO_\pf^{\times}$ defined by \begin{itemize}
\item $\chi_{\pf,1}(\varpi_\pf) = \alpha_\pf(\mathbf{N}\pf)^{-1}$ if $\pf \in T$ and $\chi_{\pf,1}(\varpi_\pf) = \beta_\pf(\mathbf{N}\pf)^{-1}$ otherwise
\item $\chi_{\pf,2}(\varpi_\pf) = \beta_\pf$ if $\pf \in T$ and $\chi_{\pf,1}(\varpi_\pf) = \alpha_\pf$ otherwise.
\end{itemize}

\begin{proposition}\label{KLapprox}
For each $S\subset\{\pf|p\}$ there is a point $x_{S}$ of $\mathscr{E}(\n,\q\lf)_{\rhobar}$ such that \[\nu(x_{S}) = \nu_{S},\] and the character \[\psi_{x_{S}}: \mathscr{H}(\n\q\lf)\rightarrow\overline{\Q}_p \] induced by $\psi$ is equal to $\theta_{f_{S}}$.
\end{proposition}
\begin{proof}
For this result, we need to use an alternative construction of the eigenvariety $\mathscr{E}(\n,\q\lf)_{\rhobar}$. This is given by the space $\mathcal{C}_\kappa(\mathfrak{m})$ of \cite[Theorem 4.5.4]{KL} (with modified level structures). Here the $\mathfrak{m}$ corresponds to our choice of residual Galois representation $\rhobar$. To show that $\mathcal{C}_\kappa(\mathfrak{m})$ coincides with $\mathscr{E}(\n,\q\lf)_{\rhobar}$ we need to verify that the subset of $\mathcal{C}_\kappa(\mathfrak{m})$ corresponding to the `classical points' $\mathscr{Z}(\n,\q\lf)_{\rhobar}$ is Zariski dense and accumulation. This follows from a classicality criterion for overconvergent Hilbert modular forms, which has recently been proved in two different ways --- by Pilloni and Stroh \cite[Th\'{e}or\`{e}me 1.2]{PS-class} and by Tian and Xiao \cite[Proposition 6.3]{TX}. Note that Theorem 6.5 of loc. cit. is the statement that the classical points are Zariski dense in the Kisin--Lai eigenvarieties, but the accumulation property also follows immediately from their proof. The existence of the points $x_{S}$ is now immediate from the construction of the Kisin--Lai eigenvariety.

Alternatively, one can avoid the appeal to difficult classicality theorems and instead show that $\mathscr{E}(\n,\q\lf)_{\rhobar}$ is isomorphic (with its additional structures) to the Zariski closure of the classical points in $\mathcal{C}_\kappa(\mathfrak{m})$, by working with the set-up described in Remark \ref{remgalalt}. Applying \cite[Theorem 4.5.6]{KL} then concludes the proof.
\end{proof}

Since the local factor $\pi_{f,\q}$ is an unramified twist of Steinberg, one naturally expects that the points $x_{S}$ lie in the eigenvariety $\mathscr{E}(\n,\q\lf)_{\rhobar}^{\q\mhyphen sp}$. Since we do not yet know local--global compatibility at $\q$ for $\pi_f$  it is a little delicate to prove this, but it follows from an application of the main result of \cite{mehilbert}.
\begin{proposition}\label{DT}
The points $x_{S}$ of Proposition \ref{KLapprox} all lie in $\mathscr{E}(\n,\q\lf)_{\rhobar}^{\q\mhyphen sp}$.
\end{proposition}
\begin{proof}
Take a point $x_{S}$ and suppose that it does not lie in $\mathscr{E}(\n,\q\lf)_{\rhobar}^{\q\mhyphen sp}$. Then $x_{S}$ is a point of $\mathscr{E}(\n,\q\lf)_{\rhobar}^{\q\mhyphen ps}$, so we may consider the image of $x_{S}$ in $\mathscr{E}(\n,\lf)_{\rhobar}$. This point satisfies the assumptions of \cite[Theorem 4.3]{mehilbert}, but the conclusion of this Theorem tells us that $x_{S}$ is indeed in $\mathscr{E}(\n,\q\lf)_{\rhobar}^{\q\mhyphen sp}$.
\end{proof}

\section{Completed cohomology and level optimisation}\label{cc}
We describe an alternate construction of the eigenvariety $\mathscr{E}(\n,\m)_{\rhobar}^{\q\mhyphen sp}$, using the completed cohomology of Shimura curves. We will then apply the results of \cite{memathann} to deduce the following result:

\begin{theorem}\label{levellower}
Let $S\subset\{\pf|p\}$ and let $x_{S}$ be the point of $\mathscr{E}(\n,\q\lf)_{\rhobar}^{\q\mhyphen sp}$ obtained from Propositions \ref{KLapprox} and \ref{DT}. Then there is a point $y_{S}$ of $\mathscr{E}(\n,\q)_{\rhobar}^{\q\mhyphen sp}$ such that
\begin{itemize} 
\item The Hecke eigenvalues outside $\lf$ coincide with those of $x_{S}$
\item $\psi_{x_{S}}(U_\lf)$ is one of the roots of $X^2-\psi_{y_{S}}(T_\lf)X + (\mathbf{N}\lf)^{1-w}\iota^{-1}\chi_f(\varpi_\lf)$.
\end{itemize}
\end{theorem}

\subsection{Completed cohomology of Shimura curves}\label{emerton}
For this section we fix a quaternion algebra $B/F$ such that $B$ is non-split at precisely one infinite place (denoted $\tau_1$) and one finite place, $\q$ (recall that $[F:\Q]$ is assumed to be even, so such quaternion algebras exist). Denote by $G_B$ the reductive algebraic group over $\Q$ arising from the unit group of $B$. Note that $G_B$ is an inner form of $\mathrm{Res}_{F/\Q}(\GL_2)$. For $U$ a compact open subgroup of $G_B(\A_f)$ we have a complex (disconnected) Shimura curve

$$M(U)(\C) = G_B(\Q)\backslash G_B(\A_f)\times (\C-\R)/U$$
where $G_B(\Q)$ acts on $\C-\R$ via the $\tau_1$ factor of $G_B(\R)$.

These curves have canonical models over $F$, which we denote by $M(U)$. We follow the conventions of \cite{carmauv} to define this canonical model.  

\begin{definition}
We define $\widetilde{H}^1(U^p,L)$ to be 
\[\left(\varprojlim_{n} \varinjlim_{U_p} H^1_{\acute{e}t}(M(U_p U^p)_{\overline{F}},\OO_L/\m_L^n)\right)\otimes_{\OO_L}L,\]
where $U^p$ is any compact open subgroup of $G_B(\A^{\infty,p})$ and $U_p$ runs over the compact open subgroups of $G_B(\Q_p)$.
\end{definition}

The $L$-vector space $\widetilde{H}^1(U^p,L)$ is naturally an $L$-Banach space with an admissible continuous action of $G_B(\Q_p)\cong \prod_{\pf|p}\GL_2(F_\pf)$. Moreover, there is a direct summand $\widetilde{H}^1(U^p,L)_{\rhobar}$ such that (in particular) all the systems of Hecke eigenvalues arising from $\widetilde{H}^1(U^p,L)_{\rhobar}$ correspond to Galois representations whose residual representation is isomorphic to $\rhobar$.

We now explain how the systems of Hecke eigenvalues parameterised by the set $\mathscr{Z}(\n,\m)_{\rhobar}^{\q\mhyphen sp}$ can be seen in the space $\widetilde{H}^1(U^p,L)_{\rhobar}$.

Suppose we have a $\Sigma$-tuple of integers $k = (k_\tau)_{\tau \in \Sigma}$ with each $k_\tau\ge 2$ and an integer $w$ with $k_\tau = w$ mod $2$ for all $\tau$.

We denote by $W_{k,w}$ the $L$-representation of $\GL_2(F_p)$ defined by
\[\otimes_{\tau\in\Sigma_p}(\tau\circ\mathrm{det})^{(w-k_\tau+2)/2}\mathrm{Sym}^{k_\tau-2} V_\tau\]
where $V_\tau$ is the representation of $\GL_2(F_p)$ acting via $\tau$ and the standard representation of $\GL_2(L)$. These representations then give rise to lisse \'{e}tale $L$-sheaves $\FF_{k,w}$ on the curves $M(U)$ (see, for example, \cite[\S 3.2]{memathann}). 

We set $U^p$ to be the prime to $p$ part of the compact open subgroup of $G_B(\A^\infty)$ given by $U_1(\n)\cap U_0(\m/\q)$. Now the Hecke algebra $\cH(\m\n)$ acts on $\widetilde{H}^1(U^p,L)$ as follows: for places $v$ prime to $\q$, we have a standard action by double coset operators associated to our fixed uniformisers $\varpi_v$. For the place $\q$, we choose a uniformiser $\varpi_{D_\q}$ of the order $\OO_{D_\q}$ whose reduced norm is equal to the fixed uniformiser $\varpi_\q$ of $\OO_\q$, and let $U_\q$ act on $\widetilde{H}^1(U^p,L)$ via the action of $\varpi_{D_\q}$. This definition is explained by the following:
\begin{lemma}
Let $\pi_\q = \mathrm{St}\otimes\mu$ be an unramified twist of the Steinberg representation of $\GL_2(F_\q)$. The local Jacquet--Langlands correspondent $\mathrm{JL}(\pi_\q)$ of $\pi_\q$ is the one-dimensional representation of $D_\q^\times$ given by $\mu\circ\mathrm{Nrd}$, where $\mathrm{Nrd}$ denotes the reduced norm. The $U_\q$-eigenvalue of the space of Iwahori-invariants in $\pi_\q$ is equal to $\mu(\varpi_\q)$, and is therefore equal to the eigenvalue for the action of $\varpi_{D_\q}$ on $\mathrm{JL}(\pi_\q)$.
\end{lemma}
\begin{proof}
This follows the standard computation of the $U_\q$-eigenvalue of the space of Iwahori-invariants in $\pi_\q$.
\end{proof}
We have the following proposition, which is proved exactly as \cite[Theorem 5.2]{memathann}

\begin{proposition}There is a $G_B(\Q_p)$, $\Gal(\overline{F}/F)$ and Hecke--equivariant isomorphism
\[\bigoplus_{(k,w)}\varinjlim_{U_p} H^1_{\acute{e}t}(M(U_p U^p)_{\overline{F}},\FF_{k,w})_{\rhobar}\otimes_L W_{k,w}^\vee \cong \widetilde{H}^1(U^p,L)_{\rhobar}^{alg} \]
\end{proposition}
where the right hand side is the space of locally algebraic vectors (in the sense of \cite[4.2.6]{EmAn}) in the $L$--Banach space representation $\widetilde{H}^1(U^p,L)_{\rhobar}$.

The above proposition allows us to determine the contribution of classical automorphic representations to the Jacquet module (in the sense of \cite{MR2292633}) $J_B(\widetilde{H}^1(U^p,L)_{\rhobar})$ of the ($\mathbb{Q}_p$-)locally analytic vectors in $\widetilde{H}^1(U^p,L)_{\rhobar}$:

\begin{lemma}
There is a $T$, $\Gal(\overline{F}/F)$ and Hecke equivariant embedding
\[\bigoplus_{(k,w)}\left(\varinjlim_{U_p} H^1_{\acute{e}t}(M(U_p U^p)_{\overline{F}},\FF_{k,w})_{\rhobar}\right )_N\otimes_L \chi_{k,w}\hookrightarrow J_B(\widetilde{H}^1(U^p,L)_{\rhobar})\]
where the subscript $N$ denotes coinvariants (i.e. the classical Jacquet module) and $\chi_{k,w}$ is the character of $T$ given by\[\begin{pmatrix}
s_1 & 0\\0 & s_2
\end{pmatrix}\mapsto \prod_{\tau\in\Sigma_p} \tau(s_1)^{(k_\tau-2-w)/2} \tau(s_2)^{(-w-k_\tau+2)/2}\]
\end{lemma}
\begin{proof}
This follows from left exactness of the Jacquet module functor and \cite[Proposition 4.3.6]{MR2292633}, since the highest weight space $(W_{k,w}^\vee)^N$ has $T$ action given by $\chi_{k,w}$.
\end{proof}

The following lemma is a standard result in the smooth representation theory of the groups $\GL_2(F_\pf)$.
\begin{lemma}\label{jacquet}Let $\mu,\mu_1,\mu_2$ be smooth complex characters of $F_\pf^\times$. 
\begin{enumerate}
\item The Jacquet module $\pi(\mu_1,\mu_2)_{N(\pf)}$ is isomorphic as a $T_\pf$ representation to \[\mu_1|\cdot|_\pf^{1/2}\otimes\mu_2|\cdot|_\pf^{-1/2} \oplus\mu_2|\cdot|_\pf^{1/2}\otimes\mu_1|\cdot|_\pf^{-1/2}.\]

\item The Jacquet module $(\mathrm{St}\otimes\mu)_{N(\pf)}$ is isomorphic as a $T_\pf$ representation to \[\mu|\cdot|_\pf\otimes\mu|\cdot|_\pf^{-1}.\]

\end{enumerate}
\end{lemma}
\begin{proof}
See for example \cite[Theorem 8.12.15]{GoldfeldBook}.
\end{proof}

As a consequence of Lemma \ref{jacquet}, together with the Jacquet--Langlands correspondence and the contribution of automorphic representations of $G_B(\A)$ to the cohomology of the curves $M(U)$, we obtain the following proposition:
\begin{lemma}
Suppose $(\psi,\nu)\in\mathscr{Z}(\n,\m)^{\q\mhyphen sp}_{\rhobar}$. Then there is a non-zero element \[v\in J_B(\widetilde{H}^1(U^p,L)_{\rhobar})\otimes_L \overline{\Q}_p\] on which the Hecke operators away from $p$ act via the character $\psi$ and on which the torus $T$ acts via the character $\nu$.
\end{lemma}
The above proposition tells us that the `classical set' $\mathscr{Z}(\n,\m)^{\q\mhyphen sp}_{\rhobar}$ can be seen in the locally analytic $T$-representations $J_B(\widetilde{H}^1(U^p,L)_{\rhobar})$. We now summarise Emerton's construction of an eigenvariety from this representation, and show that it is an eigenvariety for the triple  $(\cH(\m\n),\mathscr{Z}(\n,\m)^{\q\mhyphen sp}_{\rhobar},\mathscr{W}_\kappa)$.

The $T$-representation $J_B(\widetilde{H}^1(U^p,L)_{\rhobar})$ is naturally dual to a coherent sheaf $\mathscr{M}$ on $\widehat{T}$ (see \cite[Proposition 2.3.2]{Emint}). Denote by $Y$ the fiber product $\widehat{T}\times_{\widehat{T}_0}\mathscr{W}_\kappa$, and let $\mathscr{M}_Y$ denote the pullback of $\mathscr{M}$ to a coherent sheaf on $Y$. 

Taking the relative spectrum of the commutative subalgebra of endomorphisms of this sheaf generated by the Hecke algebra $\mathcal{H}(\m\n)$ gives a rigid space with a finite map to $Y$. Passing to the nilreduction gives a reduced rigid space which we denote by $\mathscr{E}_\kappa$. By the above lemma, we have a subset $Z \subset \mathscr{E}_\kappa$ of classical points corresponding to the elements of $\mathscr{Z}(\n,\m)^{\q\mhyphen sp}_{\rhobar}$.

\begin{lemma}\label{Emertoneig}
The space $\mathscr{E}_\kappa$, together with the classical subset $Z$, is an eigenvariety for the triple  $$(\cH(\m\n),\mathscr{Z}(\n,\m)^{\q\mhyphen sp}_{\rhobar},\mathscr{W}_\kappa).$$
\end{lemma}
\begin{proof}
The only condition we have to check is that $Z$ is an accumulation and Zariski dense subset of $\mathscr{E}_\kappa$ --- everything else follows from the construction of $\mathscr{E}_\kappa$. To prove this we have to interpret $\mathscr{E}_\kappa$ as part of (the nilreduction of) an eigenvariety constructed as in \cite{Bu2}, following \cite[Proposition 4.2.36]{MR2292633}, \cite[Lemma 5.13]{memathann}. The proof is slightly involved, the main reason being that we can show that eigenvarieties constructed with completed cohomology have nice properties only (a priori) after composing the map to weight space $\widehat{T}_0$ with a map corresponding to restriction to a finite index subgroup of $T_0$. This comes about because \cite[Proposition 4.2.36]{MR2292633} applies to `cofree' modules over an Iwasawa algebra, not `coprojective' modules.

Choose $U_p \subset G_B(\Q_p)$ a sufficiently small compact open subgroup such that $\widetilde{H}^1(U^p,L)_{\rhobar}$ is a cofree representation of $U_p/\overline{F^\times\cap U_pU^p}$, in the sense of \cite[Definition 5.7]{memathann}. Such a $U_p$ exists by \cite[Corollary 5.8]{memathann} --- in fact, it suffices to take $U_p$ pro-$p$, since our assumptions on the tame level already ensure that $U_pU^p$ is neat. 

Denoting the closed subgroup $\overline{F^\times\cap U_pU^p}$ of $U_p$ by $X$, we define a compact commutative $p$-adic analytic group by $S:= T_0 \cap U_p /X$, and denote the rigid space parameterising its continuous characters by $\widehat{S}$. The characters corresponding to points of $\mathscr{W}_\kappa$ are trivial on the units of $F$ with norm $1$, so (possibly shrinking $U_p$ if $p=2$), these characters are trivial on $X$. Therefore restriction to $T_0 \cap U_p$ gives a map $\mathscr{W}_\kappa \rightarrow \widehat{S}$. In fact, the definition of $\mathscr{W}_\kappa$ implies that this map is an isomorphism onto its image, which we denote by $\mathscr{W}_S$.

We denote by $\widetilde{\mathscr{W}}_\kappa$ the pre-image of $\mathscr{W}_S$ in $\mathscr{W}$. The space $\widetilde{\mathscr{W}}_\kappa$ is a finite disjoint union of open discs, whose components are indexed by characters of the finite group $T_0/T_0\cap U_p$. We denote by $\widetilde{Y}$ the fiber product $\widehat{T}\times_{\widehat{T}_0}\widetilde{\mathscr{W}}_\kappa$, and let $\mathscr{M}_{\widetilde{Y}}$ denote the pullback of $\mathscr{M}$ to a coherent sheaf on $\widetilde{Y}$. Mimicking the construction of $\mathscr{E}_\kappa$, we obtain a rigid space $\widetilde{\mathscr{E}}_\kappa$ with a finite map to $\widetilde{Y}$, such that $\mathscr{E}_\kappa$ is the open and closed subspace of $\widetilde{\mathscr{E}}_\kappa$ lying over $\mathscr{W}_\kappa$.

Consider an increasing sequence $X_n=\mathrm{Sp}(A_n)$ of admissible affinoid opens covering $\widehat{S}$, write $M$ for the space of global sections of $\mathscr{M}$ and $M_n$ for the base change $M\widehat{\otimes}_{\mathscr{C}^{an}(\widehat{S},L)}A_n$. It follows from \cite[Proposition 4.2.36]{MR2292633} that $M_n$ is the finite slope part of an orthonormalisable Banach $A_n$-module with the action of a compact operator (coming from the action of an element $z \in T$). We may choose the $X_n$ such that their inverse images in $\widetilde{\mathscr{W}}_\kappa$ are admissible affinoid opens $\widetilde{Y}_n=\mathrm{Sp}(B_n)$ (e.g. closed discs). It follows from \cite[Corollary 2.10]{Bu2} that the modules $M\widehat{\otimes}_{\mathscr{C}^{an}(\widehat{S},L)}B_n$ are likewise finite slope parts of orthonormalisable Banach $B_n$-modules with the action of a compact operator.

It now follows, as in \cite[Corollaire 6.4.4]{Chenun}, that the image of each irreducible component of  $\widetilde{\mathscr{E}}_\kappa$ in $\mathscr{W}_S$ is the image of a Fredholm hypersurface, and is therefore Zariski open in $\mathscr{W}_S$. For each irreducible component of $\widetilde{\mathscr{E}}_\kappa$, the map to $\mathscr{W}_S$ factors through one of the connected components of $\widetilde{\mathscr{W}}_\kappa$, so each irreducible component of $\widetilde{\mathscr{E}}_\kappa$ has Zariski open image in this connected component. In particular, each irreducible component of $\mathscr{E}_\kappa$ has Zariski open image in $\mathscr{W}_\kappa$.

The classicality criterion of \cite[Theorem 4.4.5]{MR2292633} now shows that $Z$ is Zariski dense and accumulation in $\mathscr{E}_\kappa$. 
\end{proof}

Now we know that $\mathscr{E}(\n,\m)_{\rhobar}^{\q\mhyphen sp}$ and $\mathscr{E}_\kappa$ are eigenvarieties for the same triple, we have the following:
\begin{corollary}\label{cohodesc}
The closed points $x \in \mathscr{E}(\n,\m)_{\rhobar}^{\q\mhyphen sp}$ with $\nu(x)=\lambda \in \widehat{T}(\overline{\Q}_p)$ correspond bijectively with systems of Hecke eigenvalues appearing in the (finite--dimensional) $\overline{\Q}_p$-vector space
\[J_B(\widetilde{H}^1(U^p,L)_{\rhobar})\otimes_{L}\overline{\Q}_p[\lambda]\] defined to be the subspace where $T$ acts via the character $\lambda$.
\end{corollary}
\begin{proof}
This follows from the construction of $\mathscr{E}_\kappa$ and Lemma \ref{Emertoneig}.
\end{proof}
\section{Proof of Theorem \ref{mylgc}}\label{end}

We can also construct $\mathscr{E}(\n,\q)_{\rhobar}$ using the overconvergent Hilbert modular forms of \cite{KL}. Therefore we conclude from Theorem \ref{levellower} that there are overconvergent Hilbert modular eigenforms $g_{S} \in M_\kappa^\dagger(\n,\q)$ whose systems of Hecke eigenvalues correspond to those of the points $y_{S}$. 

Applying \cite[Theorem 7.10]{kassaei}, we glue the $g_{S}$ into a classical Hilbert modular eigenform $$g \in H^0(\mathscr{M}(\n,\q)_L,\underline{\omega}^{\kappa})$$
as in the statement of Proposition \ref{proplgc} (we are using more general tame levels than Kassaei, but this presents no problem). This completes the proof of Proposition \ref{proplgc}, and hence of Theorem \ref{mylgc}.
\section{Acknowledgments}
The author would like to thank F. Calegari, R. Moy and J. Specter for correspondence relating to the work \cite{MoySpe}. The author is supported by Trinity College, Cambridge, and the Engineering and Physical Sciences Research Council.

\end{document}